\documentclass[aop,MSNbibl,citesort,seceqn,dvips]{arximspdf}
\usepackage{mathrsfs}

\doi{10.1214/10-AOP626}
\volume{40}
\issue{2}
\pubyear{2012}
\firstpage{788}
\lastpage{812}

\makeatletter

\newcommand{\arxivurl}[1]{\href{http://arxiv.org/abs/#1}{#1}}

\newcommand{\fraca}[2]{{#1/#2}}

\newcommand{\fracc}[2]{{#1/(#2)}}
\newcommand{\fracd}[2]{{(#1/#2)}}

\newcommand{\eqref}[1]{(\ref{#1})}
\newtheorem{theorem}{Theorem}[section]
\newproclaim{remark}[theorem]{Remark}
\newproclaim{example}[theorem]{Example}
\newtheorem{lemma}[theorem]{Lemma}
\newtheorem{proposition}[theorem]{Proposition}

\def\R{{\mathbb R}}
\def\C{{\mathbb C}}

\newcommand{\E}{{\mathbb E}}
\renewcommand{\P}{{\mathbb P}}
\newcommand{\F}{{\mathscr F}}

\renewcommand{\a}{\alpha}
\newcommand{\g}{\gamma}
\renewcommand{\d}{\delta}
\renewcommand{\l}{\lambda}
\renewcommand{\O}{\Omega}

\newcommand{\calL}{{\mathscr L}}
\newcommand{\one}{{\mathbf{1}}}
\newcommand{\calA}{\mathscr{A}}
\makeatother

\begin{document}
\begin{frontmatter}

\title{Stochastic maximal $L^{p}$-regularity}
\runtitle{Stochastic maximal $L^{p}$-regularity}

\begin{aug}
\author[A]{\fnms{Jan} \snm{van Neerven}\thanksref{t1}\ead[label=e1]{J.M.A.M.vanNeerven@tudelft.nl}},
\author[A]{\fnms{Mark} \snm{Veraar}\corref{}\thanksref{t2}\ead[label=e2]{M.C.Veraar@tudelft.nl}}
\and
\author[B]{\fnms{Lutz} \snm{Weis}\thanksref{t3}\ead[label=e3]{weis@math.uka.de}}
\runauthor{J. van Neerven, M. Veraar and L. Weis}
\affiliation{Delft University of Technology,
Delft University of Technology
and~Karlsruhe~Institute~of Technology}
\address[A]{J. van Neerven\\
M. Veraar\\
Delft Institute of Applied Mathematics\\
Delft University of Technology \\
P.O. Box 5031\\
2600 GA Delft\\
The Netherlands\\
\printead{e1}\\
\hphantom{E-mail: }\printead*{e2}} %adresu isvedimo komanda gale!
\address[B]{L. Weis\\
Institut f\"{u}r Analysis \\
Karlsruhe Institute of Technology \\
D-76128 Karlsruhe\\
Germany\\
\printead{e3}}
\end{aug}
\thankstext{t1}{Supported by VICI subsidy 639.033.604 of the
Netherlands Organisation for Scientific Research (NWO).}
\thankstext{t2}{Supported by the Alexander von Humboldt foundation and
VENI subsidy 639.031.930
of the Netherlands Organisation for Scientific Research (NWO).}
\thankstext{t3}{Supported by   Deutsche
Forschungsgemeinschaft Grant We 2847/1-2.}

% HISTORY:
\received{\smonth{6} \syear{2010}}
\revised{\smonth{10} \syear{2010}}

% ABSTRACT
%
\begin{abstract}
In this article we prove a maximal $L^p$-regularity result for stochastic
convolutions, which extends Krylov's basic mixed $L^p(L^q)$-inequali\-ty
for the Laplace operator on $\R^d$ to large classes of elliptic
operators, both
on $\R^d$ and on bounded domains in $\R^d$ with various boundary
conditions. Our
method of proof is based on McIntosh's $H^\infty$-functional calculus,
$R$-boundedness techniques and sharp $L^p(L^q)$-square function
estimates for stochastic
integrals in $L^q$-spaces. Under an additional invertibility assumption
on $A$,
a maximal space--time $L^p$-regularity result is obtained as well.
\end{abstract}

% KEYWORDS
%

\end{frontmatter}

%s1 ###
\section{Introduction}\label{sec1}

Let $S = (S(t))_{t\ge0}$ denote the heat semigroup on $L^p(\R^d)$,
\[
S(t)f(x) = \frac1{\sqrt{(2\pi t)^d}}\int_{\R^d}
e^{-(x-y)^2/2t}f(y)\,dy,
\]
and let $H$ be a Hilbert space. Generalizing the classical Littlewood--Paley
inequality, Krylov \cite{Kry94,Kry00,Kry06} proved that for $p\in
[2,\infty)$
and all $G\in L^p(\R_+\times\R^d; H)$ one has
%
%e1.1 ###
\begin{eqnarray}\label{eq:Kry}
&&\int_0^\infty\!\!\int_{\R^d} \biggl (\int_0^t \| [\nabla
S(t-s)G(s,\cdot)](x)\| ^2\,ds  \biggr)^{p/2}\,dx \,dt
\nonumber
\\[-8pt]
\\[-8pt] && \qquad  \le C_p^p \|
G\| _{L^p(\R_+\times\R^d; H)}^p
\nonumber
\end{eqnarray}
and more generally, for $p,q\in[2,\infty)$ with $q\leq p$
and $G\in L^p(\R_+; L^q(\R^d; H))$,
%
%e1.2 ###
\begin{eqnarray}\label{eq:Kry-pq}
&&\int_0^\infty \biggl(\int_{\R^d}  \biggl(\int_0^t \| [\nabla
S(t-s)G(s,\cdot)](x)\| ^2\,ds  \biggr)^{q/2}\,dx \biggr)^{p/q}\,dt
\nonumber
\\[-8pt]
\\[-8pt]
& & \qquad \le C_{p,q}^p \|
G\| _{L^p(\R_+; L^q(\R^d; H))}^p.
\nonumber
\end{eqnarray}
In both \eqref{eq:Kry} and \eqref{eq:Kry-pq} we implicitly use the
extension of $S(t)$ to $L^p(\R^d; H)$ and $L^q(\R^d; H)$, respectively
(see the remarks preceding Theorem \ref{thm:maxregLp}).
These singular convolution estimates are the cornerstone of Krylov's
$L^q$-theory of stochastic PDEs \cite{Kry94,Kry96,Kry,Kry00,Kry06,KryOverview}.
The proofs of \eqref{eq:Kry} and \eqref{eq:Kry-pq} rely heavily on
techniques from harmonic analysis, and their extension to bounded domains
is a well-known open problem. The aim of the present paper is to prove
a far-reaching generalization
of Krylov's inequalities which, among other things, provides such an extension.
Our approach is radically different from Krylov's and uses
$H^\infty$-calculus estimates, developed by McIntosh and coauthors,
$R$-boundedness techniques and sharp $L^p(L^q)$-square function
estimates for stochastic integrals in $L^q$-spaces.

In order to state the main result we need to introduce some terminology.
Let $(\Omega,\mathscr{A},\P)$ be a probability space endowed with a
filtration
$\F= (\F_t)_{t\ge0}$,
and let $(W_H(t))_{t\geq0}$ be a cylindrical $\F$-Brownian motion on
$H$ (see Section~\ref{sec:stochint}). Furthermore let $(\mathcal{O},\Sigma,\mu)$
be an arbitrary $\sigma$-finite measure space.

\begin{theorem}\label{thm:main} Let $q\in[2,\infty)$,
suppose the operator $A$ admits a
bounded $H^\infty$-calculus of angle less than $\pi/2$ on
$L^q(\mathcal{O})$, and let
$(S(t))_{t\ge0}$ denote the bounded analytic semigroup
on $L^q(\mathcal{O})$ generated by ${-}A$.
For all $\F$-adapted $G\in L^p(\R_+\times\Omega;L^q(\mathcal{O};H))$
the stochastic convolution process
\[
U(t) = \int_0^t S(t-s)G(s)\,dW_H(s), \qquad t\ge0,
\]
is well defined in $L^q(\mathcal{O})$, takes values in the fractional domain
$D(A^\fraca12)$ almost surely
and for all $2<p<\infty$ we have the stochastic maximal $L^p$-regularity
estimate
%
%e1.3 ###
\begin{equation}\label{eq:Apq}
\E\| A^{\fraca12} U\| _{L^p(\R_+;L^q(\mathcal{O}))}^p \le C^p \E\|
G\| _{L^p(\R_+;L^q(\mathcal{O};H))}^p
\end{equation}
with a constant $C$ independent of $G$. For $q=2$ this estimate also
holds with
$p=2$.
\end{theorem}

Although $U$ also belongs to $L^p((0,T)\times\Omega;L^q(\mathcal{O}))$
for all $T\in(0,\infty)$, in general it is
false that $U$ belongs to $ L^p(\R_+\times\Omega;L^q(\mathcal{O}))$
unless one makes the additional assumption that $A$ is invertible
[see Theorem \ref{thm:maxregintro}(1) with $\theta=0$].

The limiting case $p=2$ is not allowed in Theorem \ref{thm:main}
(except if $q=2$); a~counterexample is presented in
Section \ref{sec:counter}. This is rather surprising, since $p=2$ is usually
the ``easy'' case. Theorem \ref{thm:main} is new even for $q=2$ and
$p\in(2,
\infty)$.

The convolution process $U$
is the mild solution of the abstract stochastic PDE
\[
dU(t) + AU(t)\,dt = G(t)\,dW_H(t), \qquad t\ge0,
\]
and therefore Theorem \ref{thm:main}
can be interpreted as maximal $L^p$-regularity results for such equations.
As is well known \cite{Brz1,DPZ,Kry} (cf. Section
\ref{sec:discussion}), stochastic maximal regularity estimates can be combined
with fixed point arguments to
obtain existence, uniqueness and regularity results for solutions to
more general classes of nonlinear stochastic PDEs.
This approach has proved very fruitful in the setting of deterministic PDEs,
as can be seen from the surveys \cite{DHP,KuWe}.
In order to keep the present
paper at a reasonable length, such applications to stochastic PDEs
have been worked out in a separate paper \cite{NVW11}.
A generalization of estimate \eqref{eq:Kry} to the setting of
stochastic integrodifferential equations has been obtained in \cite{DeLo09};
our approach seems to be applicable in this context as well.

Let us now briefly indicate how \eqref{eq:Kry} and \eqref{eq:Kry-pq}
follow from Theorem \ref{thm:main} and how the corresponding estimates for
bounded regular domains may be deduced.
First of all, by the It\^{o} isomorphism
for $L^q(\mathcal{O})$-valued stochastic
integrals (see Section \ref{sec:stochint}),
the estimate \eqref{eq:Apq} can be rewritten as
%
%e1.4 ###
\begin{eqnarray}\label{eq:Kry-A}
&&\E\int_0^\infty   \biggl(\int_{\mathcal{O}}  \biggl(\int_0^t \|
[A^{\fraca12}
S(t-s)G(s,\cdot)](x)\|^2\,ds  \biggr)^{q/2}\,dx \biggr)^{p/q}\,dt
\nonumber
\\[-8pt]
\\[-8pt]
 && \qquad  \leq C^p \E
\|G\|_{L^p(\R_+;L^q(\mathcal{O}; H))}^p.
\nonumber
\end{eqnarray}
It is well known that the Laplace operator $-\frac12\Delta$ admits a bounded
$H^\infty$-cal\-culus on $L^q({\R^d})$, and
$D((-\Delta)^\fraca12)$ equals the Bessel potential space
${H}^{1,q}(\R^d)$ associated with $L^q(\R^d)$.
As a result, \eqref{eq:Kry-A} implies \eqref{eq:Kry-pq} without the
restriction $q\leq p$.
By the same token, the Dirichlet Laplacian $A=-\frac12\Delta_{ \mathrm{Dir}}$
on a~bounded regular domain
$D\subseteq\R^d$ has a bounded $H^\infty$-calculus on
$L^q(D)$, and via complex interpolation
(cf. \eqref{eq:complexfract}, \cite{KKW}, Lemma 9.7,
and \cite{Tr1}, Theorem~4.3.2.2) one has
%
%e1.5 ###
\begin{eqnarray}\label{eq:domainsDirichlet}
D(A^\fraca12) &=& [L^q(D),D(A)]_{\fraca12} = [L^q(D),H^{2,q}(D)\cap
H^{1,q}_0(D)]_{\fraca12}\nonumber
\\[-8pt]
\\[-8pt] &\subseteq& H^{1,q}(D)
\nonumber
\end{eqnarray}
with $H_0^{1,q}(D) = \{f\in H^{1,q}(D)\dvtx   f = 0 \mbox{ on } \partial
D\}$.
Noting that $\Delta_{\mathrm{Dir}}$ is invertible,~\eqref{eq:Apq} gives
$u\in L^p(\R_+\times\O;D(A^{\fraca12}))$. Hence by \eqref
{eq:domainsDirichlet} we obtain the estimate
%
%e1.6 ###
\begin{equation}\label{eq:Dir}
\E\| U\| _{L^p(\R_+;H^{1,q}(D))}^p \le C^p \E\|
G\| _{L^p(\R_+;L^q(D;H))}^p.
\end{equation}
A similar estimate, but only on bounded
time intervals, holds for Neumann Laplacian (see the remarks following
Theorem \ref{thm:main} and Remark \ref{rem:MR-finiteT}).

The main advantage of our approach is that it uses estimates from the
deterministic theory of partial differential equations (e.g., the
boundedness of
the $H^\infty$-calculus) directly as building blocks in the theory of
stochastic
partial differential equations. The boundedness of the $H^\infty
$-calculus is
not a very restrictive assumption; elliptic
operators typically satisfy this assumption on $L^q$-spaces in the range
$1<q<\infty$ (see Section \ref{sec:Hcalc} for a comprehensive list of
examples).
For second order elliptic operators on bounded regular domains~%
$D$, \eqref{eq:Dir} holds again under
mild regularity assumptions (see Example~\ref{ex:elliptic}).

Under the additional assumption that the operator $A$ is invertible,
Theorem \ref{thm:main} can be strengthened to a maximal space--time
$L^p$-regularity
result; by a standard interpolation argument this also gives a sharp
maximal inequality.

\begin{theorem}\label{thm:maxregintro}
In addition to the assumptions of Theorem \ref{thm:main} suppose that
$0\in
\varrho(A)$.
\begin{longlist}[(2)]
\item[(1)] \textup{Space--time regularity.}
For all $\theta\in[0,\frac12)$,
\[
\E\|U\|_{H^{\theta,p}(\R_+;D(A^{\fraca12-\theta}))}^p \leq C^p
  \E\|G\|_{L^p(\R_+;L^q(\mathcal{O};H))}^p.
\]
\item[(2)]\textup{Maximal estimate.}
\[
\E\sup_{t\in\R_+}\|U(t)\|_{D_A(\fraca12-\fraca1p,p)}^p \leq
C^p   \E\|G\|_{L^p(\R_+;L^q(\mathcal{O};H))}^p,
\]
where $D_A(\frac12-\frac1p,p) :=
(L^q(\mathcal{O}),D(A))_{\fraca12-\fraca1p,p}$ is the real
interpolation space.
\end{longlist}
In both cases the constant $C$ is independent of $G$.
\end{theorem}

As far as we know, Theorem \ref{thm:maxregintro} is new even for the
Laplace operator on~$L^2(\R^d)$.

The case $\theta=0$ of part (1) easily generalizes to the more general estimate
%
%e1.7 ###
\begin{equation}\label{eq:delta}
\E\|U\|_{L^p(\R_+;D(A^{\fraca12+\delta}))}^p \le C^p \E\|
G\|_{L^p(\R_+;D((A\otimes I_H)^{\delta}))}^p
\end{equation}
for any $\delta>0$.
A similar estimate for $\delta<0$ can be derived by using
extrapolation spaces, which is useful when dealing with space--time
white noise.

A further advantage of our methods is that, with the aid of some
additional tools from
functional analysis, Theorems \ref{thm:main} and \ref{thm:maxregintro}
and their proofs extend to more general function spaces, such as spaces
which are
isomorphic to closed subspaces of $L^q(\mathcal{O})$ (e.g., Sobolev
and Besov spaces).

%s1.1 ###
\subsection{Notation}\label{sec:notation}
Unless otherwise stated, all vector spaces are real.
Arguments involving spectral theory are
carried out by passing to complexifications. Throughout the paper, $H$
is a
Hilbert space, and $(\mathcal{O},\Sigma,\mu)$ is a~$\sigma$-finite
measure space.
We use the notation $(r_n)_{n\ge1}$ for a \textit{Rademacher\vadjust{\goodbreak}
sequence}, which is
a sequence of
independent random variables which take the values $\pm1$ with equal
probability.
We write $a \lesssim_k b$ to express that there exists a constant $c$, only
depending on $k$, such that $a\leq c b.$ We write $a\eqsim_k b$ to
express that
$a \lesssim_k b$ and $b \lesssim_k a$.

%s2 ###
\section{Preliminaries}\label{sec:prelim}

%s2.1 ###
\subsection{Stochastic integration}\label{sec:stochint}

Let $(\O,\calA,\P)$ be a probability space endowed with filtration
$\F= (\F_t)_{t\ge0}$.
An \textit{$\F$-cylindrical Brownian motion on $H$} is a~bounded linear
operator $\mathcal{W}_H\dvtx  L^2(\R_+;H)\to L^2(\O)$ such that:
\begin{longlist}[(iii)]
\item[(i)] for all $t\ge0$ and $h\in H$ the random variable
$W_H(t)h:= \mathcal{W}_H (\one_{(0,t]}\otimes h)$ is centred Gaussian
and $\F_t$-measurable;
\item[(ii)] for all $t_1,t_2\ge0$ and $h_1,h_2\in H$ we have
$ \E(W_H(t_1)h_1\cdot W_H(t_2)h_2) = t_1\wedge t_2 [h_1,h_2]$;
\item[(iii)]
for all $t_2\ge t_1\ge 0$ and $h\in H$ the random variable
$W_H(t_2)h - W_H(t_1)h$ is independent of $\mathscr{F}_{t_1}$.
\end{longlist}
It is easy to see that for all $h\in H$ the process $(t,\omega)\mapsto
(W_H(t) h)(\omega)$ is an $\F$-Brownian
motion (which is standard if $\| h\| =1$).

For $0\le a<b<\infty$, $\F_a$-measurable sets $F\subseteq\O$,
$h\in H$, and
$f\in L^q(\mathcal{O})$
the stochastic integral of the indicator process
$(t,\omega)\mapsto1_{(a,b]\times F}(t,
\omega) f\otimes h$ with respect to $W_H$ is defined as the
$L^q(\mathcal{O})$-valued random variable
\[
\int_0^t 1_{(a,b]\times F}  (f\otimes h)\,dW_H:= \bigl(W_H(t\wedge b) h -
W_H(t\wedge a) h\bigr)1_F f, \qquad t\ge0.
\]
By linearity, this definition extends to \textit{adapted finite rank step
processes} $G\dvtx \R_+\times\Omega\to L^p(\mathcal{O};H)$, which we
define as
finite linear combinations of
adapted indicator processes of the above form. Recall that a process
$G\dvtx \R_+\times\O\to
L^q(\mathcal{O};H)$ is called \textit{$\F$-adapted} if for every
$t\in\R_+$,
$\omega\mapsto G(t, \omega)$ is $\F_t$-measurable.

The next result is a special case of \cite{NVWco}, Theorem 6.2.

\begin{proposition}\label{prop:NVW}
Let $p\in(1, \infty)$ and $q\in(1,\infty)$ be fixed.
For all $\F$-adapted finite rank step processes $G\dvtx \R_+\times\O\to
L^q(\mathcal{O};
H)$ we have the ``It\^{o} isomorphism''
\[
c^p\E\| G\|_{L^q(\mathcal{O};L^2(\R_+;H))}^p \le\E \biggl\| \int
_0^\infty
G\,dW_H \biggr\|^p_{L^q(\mathcal{O})} \le C^p \E\|G\|_{L^q(\mathcal
{O};L^2(\R_+;H))}^p
\]
with constants $0<c\le C<\infty$ independent of $G$.
\end{proposition}

By a standard density argument,
these inequalities can be used to extend the stochastic integral to the
Banach space
$L_\F^p(\O;L^q(\mathcal{O};L^2(\R_+;H)))$ of all $\F$-adapted
processes $G\dvtx \R_+\times\O\to L^q(\mathcal{O};H)$
which belong to
$L^p(\O;L^q(\mathcal{O};\allowbreak L^2(\R_+;H)))$.\vadjust{\goodbreak}
In the remainder of this paper, all stochastic integrals are understood
in the
above sense. By Doob's inequality, the inequalities remain true if
the middle term is replaced by the corresponding maximal norm.
In this form, for $p=q$ they follow
directly from the (real-valued) Burkholder--Davis--Gundy inequality.

By Minkowski's inequality, for $q\in[2,\infty)$ one has
\[
\E\| G\|_{L^q(\mathcal{O};L^2(\R_+;H))} \le\E\| G\|_{L^2(\R
_+;L^q(\mathcal{O};H))}.
\]
In combination with Proposition \ref{prop:NVW}, for $p\in(1,\infty)$
and $q\in[2,\infty)$
this gives the one-sided inequality
%
%e2.1 ###
\begin{equation}\label{eq:Mtype2}
\E \biggl\| \int_0^\infty G\,dW_H \biggr\| _{L^q(\mathcal{O})}^p
\le C^p \E\| G\| _{L^2(\R_+;L^q(\mathcal{O};H))}^p.
\end{equation}

\begin{remark}\label{rem:UMD}
For $q\in[1,\infty)$ the space $L^q(\mathcal{O};H)$ is canonically
isomorphic to the space
$\g(H,L^q(\mathcal{O}))$ of $\g$-radonifying operators from $H$ to
$L^q(\mathcal{O})$ (see \cite{NVW3} and the references
given therein). Using this identification, Proposition \ref{prop:NVW}
extends to arbitrary
UMD Banach spaces $E$ (see \cite{NVW1}, Theorems 5.9 and 5.12); this
class of Banach spaces includes
the spaces $L^q(\mathcal{O})$ with $q\in(1,\infty)$. For Hilbert
spaces $E$ one has the further
identification $\gamma(H,E) =\calL_2(H,E)$, the space
of Hilbert--Schmidt operators from $H$ to $E$.

The inequality
\eqref{eq:Mtype2} holds for arbitrary Banach spaces $E$
with martingale type $2$ (see \cite{Brz1,Brz2});
this class of Banach spaces includes
the spaces $L^q(\mathcal{O})$ with $q\in[2,\infty)$.
\end{remark}

%s2.2 ###
\subsection{$R$-boundedness}\label{sec:gbdd}

Let $E_1$ and $E_2$ be Banach spaces, and let $(r_n)_{n\ge1}$ be a Rademacher
sequence (see Section \ref{sec:notation}).
A family $\mathscr{T}$ of bounded linear operators from $E_1$ to $E_2$
is called \textit{$R$-bounded} if there exists a constant $C\ge0$ such
that for all finite sequences $(x_n)_{n=1}^N$ in $E_1$ and
$(T_n)_{n=1}^N$ in ${\mathscr{T}}$ we have
\[
\E \Biggl\| \sum_{n=1}^N r_n T_n x_n \Biggr\| ^2
\le C^2\E \Biggl\| \sum_{n=1}^N r_n x_n \Biggr\| ^2.
\]
The least admissible constant $C$ is called the \textit{$R$-bound} of
$\mathscr{T}$, notation $R(\mathscr{T})$. For Hilbert spaces $E_1$
and $E_2$,
$R$-boundedness is
equivalent to uniform boundedness and $R(\mathscr{T}) = \sup_{t\in
\mathscr{T}} \|T\|$. The notion of
$R$-boundedness has played an important role in recent
progress in the regularity theory of (deterministic) parabolic evolution
equations. For more information on $R$-boundedness and its applications
we refer the reader to \cite{CPSW,DHP,KuWe}.

In our applications, $E_1$ and $E_2$ will always be $L^q$-spaces or
mixed $L^p(L^q)$-spaces
(possibly with values in $H$). All such spaces are examples of Banach
function spaces which are $s$-concave for some $s<\infty$ [for this
purpose we identify $H$ with $\ell^2(I)$
over a suitable index set $I$]. For these spaces, Rademacher sums
can\vadjust{\goodbreak}
be evaluated,
up to a constant, by means of \textit{square functions} (see \cite{LiTz}, Theorem~1.d.6)
\[
\Biggl (\E \Biggl\| \sum_{n=1}^N r_n x_n \Biggr\|^2_E \Biggr)^{1/2} \eqsim_E
 \Biggl\| \Biggl(\sum_{n=1}^N |x_n|^2 \Biggr)^{1/2}  \Biggr\|_E.
\]

Below we shall need a continuous version of the right-hand side,
for which we need to introduce some notation. Let $E$
be a Banach function space over $(\mathcal{O}, \Sigma, \mu)$.
For a Hilbert space $\mathcal{H}$, let $E(\mathcal{H})$ be the space
of all
strongly $\mu$-measurable functions $G\dvtx \mathcal{O}\to\mathcal{H}$
for which
$\|G(\cdot)\|_\mathcal{H}$ belongs to $E$.
Typically we shall take $\mathcal{H}= L^2(\R_+,\nu)$
with $\nu$ a $\sigma$-finite Borel measure on $\R_+$.

For $E_1=E_2=L^q(\mathcal{O})$ the next multiplier result
is due to \cite{Weis-maxreg}; the version below is included as a
special case in a
more general operator-theoretic formulation of this result,
valid for arbitrary Banach spaces $E_1$ and $E_2$, in~\cite{KaWe}
(a proof is reproduced in \cite{vNCan}).

\begin{proposition}\label{prop:KW}
Let $E_1$ and $E_2$ be Banach function spaces with finite cotype, and
let $\nu$ be a
$\sigma$-finite Borel measure on $\R_+$.
Let $M\dvtx \R_+\to\calL(E_1,E_2)$ be a function with the following
properties:
\begin{longlist}[(2)]
\item[(1)] for all $x\in E_1$ the function $t\mapsto M(t)x$ is strongly $\nu
$-measurable in $E_2$;

\item[(2)] the range $\mathscr{M} := \{M(t)\dvtx   t\in\R_+\}$ is $R$-bounded
in $\calL(E_1,E_2)$.
\end{longlist}
Then for all $G\dvtx \R_+\to E_1$ which satisfy $G\in E_1(L^2(\R_+,\nu))$
the function $MG\dvtx  \R_+\to E_2$ satisfies
$MG\in E_2(L^2(\R_+,\nu))$ and
\[
\|MG\|_{E_2(L^2(\R_+,\nu))}\le R(\mathscr{M})\|G\|_{E_1(L^2(\R
_+,\nu))}.
\]
\end{proposition}

Conversely, this multiplier property characterizes $R$-bounded
families. This fact
will not be needed here.

%s2.3 ###
\subsection{\texorpdfstring{Operators with a bounded $H^\infty$-calculus}
{Operators with a bounded H infinity-calculus}}\label{sec:Hcalc}
The $H^\infty$-calculus was originally developed by McIntosh and his
collaborators
\cite{ADM,AMcN,CDMY,McI}
in a line of research which eventually culminated in the solution of the
Kato square root problem \cite{KatoSqrt}. Meanwhile, this technique
has found widespread
applications in harmonic analysis and PDEs. For an in-depth treatment
of the theory
we refer to
\cite{Haase2,KuWe,Weis-survey}.

Let ${-}A$ be the generator of a bounded strongly continuous analytic
semigroup of
operators on a Banach space $E$.
As is well known (see \cite{Am}, Proposition~I.1.4.1),
the spectrum of $A$ is contained in the closure of a sector
$\Sigma_{\sigma_0} := \{z\in\C\setminus\{0\}\dvtx
|\arg(z)|<\sigma_0\}$ for some $\sigma_0\in(0,\frac12\pi)$,
and for all $\sigma\in(\sigma_0,\pi)$ one has
\[
\sup_{z\in\C\setminus\Sigma_{\sigma}} \| z(z-A)^{-1}\|<\infty.
\]
Let $H^\infty(\Sigma_\sigma)$ denote the
Banach space of all bounded analytic functions
$\varphi\dvtx \Sigma_\sigma\to\C$ endowed with the supremum\vadjust{\goodbreak} norm,
and let $H^\infty_0(\Sigma_\sigma)$
be its linear subspace consisting of all functions
satisfying an estimate
\[
|\varphi(z)|\leq\frac{C|z|^\varepsilon}{(1+|z|^2)^\varepsilon}\vspace*{-1pt}
\]
for some $\varepsilon>0$.
For $\varphi\in H^\infty_0(\Sigma_\sigma)$ and $\sigma'\in(\sigma
_0,\sigma)$
the Bochner integral
\[
\varphi(A) = \frac{1}{2\pi i}\int_{\partial\Sigma_{\sigma'}}
\varphi(z) (z-A)^{-1}
\, dz\vspace*{-1pt}
\]
converges absolutely and is independent of $\sigma'$.
We say that $A$ has a \textit{bounded $H^\infty(\Sigma_{\sigma})$-calculus}
if there is a constant $C\ge0$ such that
%
%e2.2 ###
\begin{equation}\label{eq:Hinfty}
\|\varphi(A)\|\leq C \|\varphi\|_\infty, \qquad\varphi\in
H^\infty_0(\Sigma_\sigma).\vspace*{-1pt}
\end{equation}
The infimum of all $\sigma$ such that $A$ admits a bounded
$H^\infty(\Sigma_{\sigma})$-calculus is called the \textit{angle}
of the
calculus.

In order to avoid unnecessary technicalities, from now on we shall
always assume that $A$ is
injective and has dense range. This hardly entails any loss of
generality; as
for generators ${-}A$ of bounded analytic semigroups on reflexive Banach
spaces $E$
one has a direct sum decomposition
$ E = {\mathsf N}(A) \oplus\overline{{\mathsf R}(A)}$ into kernel and
closure of the range of $A$
(see, e.g., \cite{KuWe}). In particular, such an operator
is the direct sum of a zero operator and an injective sectorial~%
ope\-rator with dense range
(see Remark \ref{rem:MR-finiteT} for further discussion on this issue).

If $A$ has a bounded $H^\infty(\Sigma_{\sigma})$-calculus, the
mapping $\varphi\mapsto\varphi(A)$ has a
unique extension to a bounded homomorphism from $H^\infty(\Sigma
_{\sigma})$ to $\calL(E)$ which
satisfies \eqref{eq:Hinfty} with the same constant $C$.

Even on Hilbert spaces $E$ there exist generators
${-}A$ of bounded strongly continuous
analytic semigroups for which $A$ does not admit a bounded
$H^\infty$-calculus (see \cite{KuWe}, Example 10.17). Examples of
operators which do admit such a calculus are
collected below.

We shall need a generalization, taken from \cite{KaWe} (see also \cite
{LeM04}),
of McIntosh's square function
characterization for the boundedness of $H^\infty$-calculi in Hilbert spaces
\cite{McI} (see also \cite{CDMY}).
We use the notation of Proposition \ref{prop:KW}
with $d\nu= \frac{dt}{t}$.\vspace*{-3pt}

\begin{proposition}\label{prop:Hinftysquare}
Let $E = L^q(\mathcal{O})$ with $q\in(1, \infty)$. Assume that $A$
has a~bounded
$H^\infty(\Sigma_{\sigma})$-calculus on $E$ for some $\sigma\in
(0,\pi/2)$.
For each $\varphi\in H^\infty_0(\Sigma_{\sigma})$
there exists a constant $C\ge0$ such that
\begin{eqnarray*}
\|t\mapsto\varphi(tA) x \|_{E(L^2(\R_+,\fraca{dt}{t}))} &\leq& C
\|x\|_{E},   \qquad    x\in E,
\\[-2pt] \|t\mapsto\varphi(tA^*) x^* \|_{E^*(L^2(\R_+,\fraca{dt}{t}))}
&\leq& C
\|x^*\|_{E^*},  \qquad    x^*\in E^*.\vspace*{-2pt}
\end{eqnarray*}
\end{proposition}

Here, as before,
\[
\|t\mapsto\varphi(tA) x \|_{E(L^2(\R_+,\fraca{dt}{t}))}
=  \biggl\|  \biggl(\int_{\R_+} |\varphi(tA) x|^2 \,\frac{dt}{t}
\biggr)^{1/2}  \biggr\|_{E}\vspace*{-1pt}
\]
and similarly for the expression involving $A^*$. Proposition \ref
{prop:Hinftysquare}
actually can be extended to arbitrary angles $\sigma\in(0,\pi)$,
but we will not need this fact.\vadjust{\goodbreak}

In the converse direction, if ${-}A$ is a generator of a bounded strongly
continuous analytic semigroup on $E=L^q(\mathcal{O})$
and the above inequalities
hold for some nonzero $\varphi\in H^\infty_0(\Sigma_{\sigma})$,
then $A$ has a bounded
$H^\infty(\Sigma_{\sigma'})$-calculus for all $\sigma'>\sigma$
\cite{KaWe}.

We will also need the fact (combine \cite{Tr1}, Theorem 1.15.3, and
\cite{Haase2}, Proposition~3.1.9) that if $A$ has a bounded
$H^\infty(\Sigma_{\sigma})$-calculus for some $\sigma\in(0,\frac
12\pi)$, then
$A$ has bounded imaginary powers and
$\sup_{s\in[-1,1]}\|A^{is}\|<\infty$. In particular this implies,
for all $\theta\in(0,1)$,
%
%e2.3 ###
\begin{equation}\label{eq:complexfract}
[E,D(A)]_{\theta} = D(A^{\theta})   \qquad \mbox{with equivalent norms},\vspace*{-1pt}
\end{equation}
where $[E,D(A)]_{\theta}$ is the complex interpolation space of exponent
$\theta$.\vspace*{-3pt}

%s2.3.1 ###
\subsubsection{\texorpdfstring{Examples of operators with a bounded $H^\infty$-calculus}
{Examples of operators with a bounded H infinity-calculus}}
\label{subsec:examplH}
Many common differential operators are known to admit a bounded
$H^\infty$-calculus
(see, e.g., the lecture notes \cite{DHP,KuWe} and the survey article
\cite{Weis-survey}). In this paragraph we collect some examples
illustrating this point.
We always take $q\in(1,\infty)$.\vspace*{-1pt}

\begin{example}\label{ex:Laplace}
The most basic example is the Laplace operator $A = -\frac12\Delta$
on $L^q(\R^d)$,
which has a bounded $H^\infty$-calculus of zero angle; this follows
from an
application of the Mihlin multiplier theorem (see \cite{KuWe}, Example~10.2b).
For this operator one has $D(A^{\fraca12}) = H^{1,q}(\R^d)$.

Also the Laplace operator with Dirichlet boundary conditions on $L^q(D)$,
where $D\subseteq\R^d$ is a bounded domain with $C^2$-boundary, has a bounded
$H^\infty$-calculus of zero angle (see \cite{DDHPV}).
In this case
one has $D(A^{\fraca12}) \subseteq H^{1,q}(D)$
[with equality if we replace $H^{1,q}(D)$ by
$H^{1,q}_0(D)= \{f\in H^{1,q}(D)\dvtx   f=0 \mbox{ on }\partial D\}$
(see
\cite{Amnonhom}, Remark~7.3,
combined with \cite{DDHPV}, Theorem~2.3, for $C^2$-domains)].

Similar results hold under different boundary conditions.\vspace*{-1pt}
\end{example}

\begin{example}\label{ex:elliptic}
Let $D\subseteq\R^d$ be a bounded domain with $C^2$-boundary.
Consider the closed and densely defined operator $A$ in $L^q(D)$ defined
by
\[
-A f(x) = \sum_{i,j=1}^d a_{ij}(x)\,\partial_{x_i x_j}f(x)+ \sum_{j=1}^d
b_j(x)\,\partial_{x_j} f(x) + c(x) f(x)\vspace*{-1pt}
\]
on the domain $D(A) = H^{2,q}(D)\cap H^{1,q}_0(D)$.
We assume that the coefficient $a_{ij}=a_{ji}$ and $b_j, c_j$ are
bounded and measurable and that
${-}A$ is uniformly elliptic, that is, there is a constant $\nu>0$ such that
\[
\sum_{i,j=1}^d a_{ij}(x) \xi_i \xi_j \geq\nu|\xi|^2, \qquad     x\in
D,\
 \xi\in\R^d.\vspace*{-1pt}
\]
It is shown in \cite{AHS,DDHPV} that if the coefficients $a_{ij}$ are
H\"{o}lder
conditions on~$\overline{D}$, then $w+A$ admits a bounded $H^\infty$-calculus
of angle less than $\pi/2$ for all $w\in\R$ large enough, and one has
$D((w+A)^{\fraca12}) \subseteq H^{1,q}(D)$.
An analogous result holds for $D=\R^d$; in this case one can weaken
the H\"{o}lder
continuity assumption on $a_{ij}$ to a VMO assumption (see \cite{DuYa}).\vadjust{\goodbreak}

Similar results hold for higher-order parameter-elliptic systems
on smooth domains satisfying the Lopatinksii--Shapiro conditions
(see \cite{DDHPV}, Theorem~2.3, and \cite{KKW}).\vspace*{-2pt}
\end{example}

\begin{example}\label{ex:Stokes}
Let $D\subseteq\R^d$ be a bounded domain with $C^2$-boundary. An
important operator
arising in the context of the Navier--Stokes equations is the Stokes operator
$A = -P \Delta$, where $P$ is the Helmholtz projection of
$[L^q(D)]^d$ onto the Helmholtz space $L^q_{\sigma}(D)$,
with domain $D(A) = [H^{2,q}(D)]^d\cap[H^{1,q}_0(D)]^d \cap
L^q_{\sigma}(D)$. The
operator $A$
has a bounded $H^\infty$-calculus of angle less than $\pi/2$ on $
L^q_{\sigma}(D)$ (see
\cite{KKW}, Theorem 9.17, and the references therein).
For $w\in\R$ large enough $D((w+A)^{\fraca12}) = [H^{1,q}_0(D)]^d\cap
L^q_{\sigma}(D)$.\vspace*{-2pt}
\end{example}

\begin{example}\label{ex:contraction}
Let ${-}A$ be an injective operator with dense range which generates a
positive contraction semigroup $S=(S(t))_{t\geq0}$ on
$L^q(\mathcal{O})$.
If $S$ extends to a bounded analytic semigroup on $L^q(\mathcal{O})$,
then $A$ has a bounded
$H^\infty$-calculus of angle less than $\pi/2$ (see \cite{KWcalc}, Corollary~5.2).\vspace*{-2pt}
\end{example}

Some of the above results have extensions to domains $D$ with
$C^{1,1}$-boundary.
Further important examples of operators with a bounded $H^\infty
$-calculus can
be obtained by considering kernels bounds (see \cite{BlKu} and the references
therein).
Finally, we note that also operators of Ornstein--Uhlenbeck type and
operators of Schr\"{o}dinger type $A=-\Delta+V$ have a
bounded $H^\infty$-calculus (see \cite{KKW}).\vspace*{-3pt}

%s3 ###
\section{$R$-boundedness of stochastic convolutions}\label{sec:Lq}
In this section we will prove the $R$-boundedness of a certain family of
stochastic convolution operators. This result plays a key role in the proofs
of Theorems \ref{thm:main} and \ref{thm:maxregintro}.

Fix
$p\in[2,\infty)$ and $q\in[2,\infty)$ for the moment, and let
$\mathcal{K}$
be the set of all absolutely continuous functions $k\dvtx \R_+\to\R$ such
that $\lim_{t\to\infty} k(t) = 0$ and
\[
\int_0^\infty\sqrt{t} | k'(t)| \, dt \leq1.\vspace*{-1pt}
\]

For $k\in\mathcal{K}$ and $\F$-adapted finite rank step processes
$G\dvtx \R_+\times\O\to L^q(\mathcal{O};H)$ we define the process $I(k)
G\dvtx \R_+\times\O\to L^q(\mathcal{O})$ by
%
%e3.1 ###
\begin{equation}\label{eq:Ikoperator}
I(k) G(t) := \int_0^t k(t-s) G(s) \, dW_H(s),  \qquad   t\ge0.\vspace*{-1pt}
\end{equation}
Since $G$ is an $\F$-adapted finite rank step process, the It\^{o}
isometry for
scalar-valued processes shows that these stochastic integrals are well defined
for all $t\ge0$. From \eqref{eq:Mtype2} it follows
that $I(k)$ extends to a bounded operator from $L^{p}_\mathscr{F}(\R
_+\times\O;L^q(\mathcal{O};H))$ into
$L^{p}(\R_+\times\O;L^q(\mathcal{O}))$, and that the family
\[
\mathcal{I} := \{I(k)\dvtx  k\in\mathcal{K}\}\vspace*{-1pt}
\]
is uniformly bounded. Indeed,
for any $k\in\mathcal{K}$ we can write
\[
k(s) = -\int_s^{\infty} k'(r) \, d r,    s\in\R_+.\vspace*{-1pt}\vadjust{\goodbreak}
\]
Now, since $G$ is an $\F$-adapted finite rank step process, the
stochastic Fubini theorem may be applied (see \cite{DPZ,NVfub}), and
for all $t>0$ we have
%
%e3.2 ###
\begin{eqnarray}\label{eq:IkJ}
I(k) G(t) &=& - \int_0^\infty k'(r) \int_0^\infty\one_{0< s <t} \one
_{t-s <
r} G(s) \, d W_H(s) \, dr
\nonumber
\\[-9pt]
\\[-9pt]
& =& - \int_0^\infty\sqrt{r} k'(r) J(r) G (t) \, dr.
\nonumber
\end{eqnarray}
Here for $r>0$ the process $J(r)
G\dvtx \R_+\times\O\to L^q(\mathcal{O})$ is defined by
\[
J(r) G(t) := \frac1{\sqrt{r}} \int_{(t-r)\vee0}^t G \, dW_H.
\]
By \eqref{eq:Mtype2} the operators $J(r)$ are bounded
from $L^{p}_\mathscr{F}(\R_+\times\O;L^q(\mathcal{O};H))$ to
$L^{p}(\R_+\times\O;L^q(\mathcal{O}))$,
and the family
\[
\mathcal{J} := \{J(r)\dvtx   r>0\}
\]
is uniformly bounded.
Now the uniform boundedness of $\mathcal{I}$ follows from
\eqref{eq:IkJ}.\vspace*{-1pt}

\begin{theorem} \label{thm:J-Lq}
For all $p\in(2,\infty)$ and $q\in[2,\infty)$ the family $\mathcal
{I}$ is
$R$-boun\-ded from
$L_\F^p(\R_+\times\O,L^q(\mathcal{O};H))$ to
$L^p(\R_+\times\O;L^q(\mathcal{O}))$. The same result holds when $p=q=2$.\vspace*{-1pt}
\end{theorem}

The case $p=q=2$ follows from the general
fact that a family of Hilbert spaces is $R$-bounded if and only if it is
uniformly bounded. In what follows we shall concentrate ourselves on the
cases $p\in(2,\infty)$ and $q\in[2,\infty)$.

By the same reasoning as before the problem of $R$-boundedness of
$\mathcal{I}$ can be reduced
to that of the family $\mathcal{J}$.\vspace*{-1pt}

\begin{proposition}\label{prop:Rboundredconvex}
If $\mathcal{J}$ is $R$-bounded, then
$\mathcal{I}$ is $R$-bounded and\break \mbox{$R(\mathcal{I})\leq
R(\mathcal{J}).$}\vspace*{-1pt}
\end{proposition}

\begin{pf}
This follows from \eqref{eq:IkJ}, convexity and density (see \cite{KuWe}, Corollary~2.14).
\end{pf}

The remainder of this section is devoted to the proof that $\mathcal
{J}$ is $R$-bounded
from $L^{p}_\mathscr{F}(\R_+\times\O;L^q(\mathcal{O};H))$ to $
L^{p}(\R_+\times\O;L^q(\mathcal{O}))$ for the indicated ranges of
$p$ and $q$.

We begin with a duality lemma which is a straightforward generalization
from the scalar
case presented in \cite{MaaNee}, Proposition 8.12. Here the absolute
values are to be taken in the pointwise sense.\vspace*{-1pt}
\begin{lemma} \label{lem:MaaNee} Let $(T(\delta))_{\delta>0}$ be a
strongly continuous one-parameter
family of positive linear operators on $L^{r}(\mathcal
{N};L^{s}(\mathcal{O}))$, where
$r,s\in[1,\infty]$ and~$(\mathcal{N},\nu)$ is another $\sigma
$-finite measure space, and
suppose the maximal function
%
%e3.3 ###
\begin{equation}\label{eq:Tstar}
T_\star(g) := \sup_{\d>0} |T(\d)g|\vspace*{-1pt}\vadjust{\goodbreak}
\end{equation}
is measurable and $L^{r}(\mathcal{N};L^{s}(\mathcal{O}))$-bounded by some
constant $C\geq0$. Let $\frac1r+\frac1{r'} = 1$, $\frac1{s}+\frac1{s'}=1$.
Then, for all $N\geq1$, $f_1,
\ldots , f_N\in L^{r'}(\mathcal{N};L^{s'}(\mathcal{O}))$ and $\d
_1,\ldots ,  \d_N>0$,
\[
 \Biggl\| \sum_{n=1}^N T^*(\d_n) |f_n|  \Biggr\|_{L^{r'}(\mathcal
{N};L^{s'}(\mathcal{O}))} \leq C  \Biggl\|
\sum_{n=1}^N |f_n|  \Biggr\|_{L^{r'}(\mathcal{N};L^{s'}(\mathcal{O}))}.\vspace*{-1pt}
\]
\end{lemma}

For functions $f\in L^r(\R_+)$ we define
the \textit{one-sided Hardy--Littlewood maximal function} $M(f)\dvtx \R
_+\to[0,\infty]$
by
\[
M(f)(t) := \sup_{\delta>0} \frac1{\delta}\int_t^{t+\delta}
|f(\tau)|\, d\tau.\vspace*{-1pt}
\]
Similarly, for functions $f\in
L^r(\R_+;L^s(\mathcal{O}))$ we define
\[
\widetilde{M}(f)(t)(a) := \sup_{\delta>0} \frac1{\delta}\int
_t^{t+\delta}
|f(\tau)(a)|\, d\tau, \qquad a\in\mathcal{O}.\vspace*{-1pt}
\]

\begin{proposition}[(Fefferman--Stein)]\label{prop:FS}
    For all $r\in
(1,\infty)$ and $s\in(1,\infty]$~the one-sided Hardy--Littlewood
maximal function
$\widetilde{M}$ is
bounded on $ L^r(\R_+;\allowbreak L^s(\mathcal{O}))$.\vspace*{-1pt}
\end{proposition}

\begin{pf}
This follows from the usual (discrete) formulation of the Feffer\-man--Stein
inequality (see \cite{Stein93}, Section II.1) by
approximation (the cases $r=s$ and $s=\infty$ are easy consequences of
the Hardy--Littlewood
maximal inequality).\vspace*{-1pt}
\end{pf}

\begin{pf*}{Proof of Theorem \ref{thm:J-Lq}}
It remains to prove the $R$-boundedness of $\mathcal{J}$.

Let $N\,{\ge}\,1$, $\delta_1,\ldots ,\delta_N\,{>}\,0$ and $G_1,\ldots
,G_N\,{\in}\,L^p_{\mathscr{F}}(\R_+\times\O;L^q(\mathcal{O};H))$ be arbi\-trary
and fixed.
Note that the functions $f_n:= \|G_n\|_H^2$ belong to
$L^{p/2}(\R_+\times\O;\break L^{q/2}(\mathcal{O}))$.

Let $(r_n)_{n=1}^N$ be a Rademacher sequence on a probability space
$(\O_r,
\F_r,\P_r)$.
Using the inequalities of Proposition \ref{prop:NVW}
applied pointwise with respect to $(\omega,t)\in\Omega_r\times\R
_+$ [in (i)]
and the Kahane--Khintchine inequalities [in (ii) and~(iii)],
we may estimate as follows (with implicit constants independent of the choice
of $N$, $\d_n$, and $G_n$):
\begin{eqnarray*}
 && \E_r \Biggl\|\sum_{n=1}^N r_n J(\delta_n)
G_n \Biggr\|_{L^p(\R_+\times\O;L^q(\mathcal{O}))}^p
\\[-2pt] && \qquad   = \E_r \Biggl\|t\mapsto\int_{0}^\infty\sum_{n=1}^N
\frac{r_n}{\sqrt{\d_n}}\one_{((t-\delta_n)\vee0, t)} G_n \, d
W_H \Biggr\|_{L^p(\R_+\times\O;L^q(\mathcal{O}))}^p
\\[-2pt] && \qquad   \stackrel{\mathrm{(i)}}{\eqsim}_{p,q} \E_r  \Biggl\|t\mapsto
\Biggl(\int_0^\infty
 \Biggl\|\sum_{n=1}^N \frac{r_n}{\sqrt{\d_n}}\\[-2pt]
 &&\hphantom{\qquad   \stackrel{\mathrm{(i)}}{\eqsim}_{p,q} \E_r  \Biggl\|t\mapsto
\Biggl(\int_0^\infty
 \Biggl\|\sum_{n=1}^N}
 {}\times\one_{((t-\delta
_n)\vee0, t)}(s)
G_n(s) \Biggr\|_H^2 \, ds  \Biggr)^{1/2}  \Biggr\|_{L^p(\R_+\times\O
;L^q(\mathcal{O}))}^p
\\[-2pt] && \qquad   = \E\int_0^\infty\E_r \Biggl\|\sum_{n=1}^N \frac{r_n}{\sqrt
{\d_n}}\one_{((t-\delta_n)\vee0, t)}
G_n \Biggr\|_{L^q(\mathcal{O};L^2(\R_+;H))}^p \, dt
\\[-2pt] && \qquad   \stackrel{\mathrm{(ii)}}{\eqsim}_{p,q} \E\int_0^\infty \Biggl(\E
_r \Biggl\|\sum_{n=1}^N \frac{r_n}{\sqrt{\d_n}}\one_{((t-\delta
_n)\vee0, t)}
G_n \Biggr\|_{L^q(\mathcal{O};L^2(\R_+;H))}^q \Biggr)^{p/q} \, dt
\\[-2pt] && \qquad   = \E\int_0^\infty \Biggl(\int_{\mathcal{O}} \E_r  \Biggl\|
\sum_{n=1}^N \frac{r_n}{\sqrt{\d_n}}\one_{((t-\delta_n)\vee0,
t)} G_n \Biggr\|_{L^2(\R_+;H)}^q\,
d\mu \Biggr)^{p/q}\, dt
\\[-2pt] && \qquad   \stackrel{\mathrm{(iii)}}{\eqsim_q} \E\int_0^\infty \Biggl(\int
_{\mathcal{O}}
 \Biggl(\E_r  \Biggl\| \sum_{n=1}^N \frac{r_n}{\sqrt{\d_n}}\one
_{((t-\delta_n)\vee0,
t)} G_n \Biggr\|_{L^2(\R_+;H)}^2 \Biggr)^{{q}/{2}}\, d\mu \Biggr)^{p/q} \, dt
\\[-2pt] && \qquad   = \E\int_0^\infty \Biggl(\int_{\mathcal{O}}  \Biggl(\int
_0^\infty\E_r \Biggl\|
\sum_{n=1}^N \frac{r_n}{\sqrt{\d_n}}\one_{((t-\delta_n)\vee0, t)}(s)
G_n(s) \Biggr\|_{H}^2\,ds \Biggr)^{q/2}\, d\mu \Biggr)^{p/q} \, dt
\\[-2pt] && \qquad   = \E\int_0^\infty \Biggl(\int_{\mathcal{O}}  \Biggl(\int
_0^\infty\sum_{n=1}^N
\frac1{\d_n}\one_{((t-\delta_n)\vee0, t)}(s) f_n(s)\,ds
\Biggr)^{{q}/{2}}\,
d\mu \Biggr)^{p/q} \, dt
\\[-2pt] && \qquad   = \E \Biggl\|\sum_{n=1}^N T^*(\delta_n) f_n \Biggr\|_{L^{p/2}(\R_+;
L^{q/2}(\mathcal{O}))}^{p/2},
\end{eqnarray*}
where the positive linear operators $T^*(\delta)$ on
$L^{p/2}(\R_+;L^{q/2}(\mathcal{O}))$ are defined by
\[
T^*(\delta) \phi(t) := \frac1{\delta}\int_{(t-\delta)\vee0}^t
\phi(s)\, ds,
\qquad\phi\in L^{p/2}(\R_+; L^{q/2}(\mathcal{O})).
\]
Let $\frac2p+\frac{1}{r}=1$ and $\frac2q+\frac{1}{s}=1$.
Then $T^*(\d)$ is the adjoint of\vspace*{1pt} the operator $T(\delta)$ on
$L^{r}(\R_+;L^s(\mathcal{O}))$ given by
\[
T(\delta)\psi(t) = \frac1{\delta}\int_t^{t+\delta} \psi(s)\, ds,
\qquad
\psi\in L^r(\R_+; L^s(\mathcal{O})).
\]
Since $\sup_{\delta>0} |T(\delta)\psi|\leq\widetilde{M}
(\psi) $ and the
latter is bounded on $L^{r}(\R_+;L^s(\mathcal{O}))$ by Proposition
\ref{prop:FS}, by Fubini's theorem we find that $T_\star$ is bounded on
$L^r(\R_+\times\O; L^s(\mathcal{O}))$. Hence we may apply Lemma
\ref{lem:MaaNee} to
conclude that
\begin{eqnarray*}
\E \Biggl\|\sum_{n=1}^N T^*(\delta_n) f_n  \Biggr\|_{L^{p/2}(\R_+;
L^{q/2}(\mathcal{O}))}^{p/2}
& \lesssim_{p,q}& \E \Biggl\|\sum_{n=1}^N f_n  \Biggr\|_{L^{p/2}(\R_+;
L^{q/2}(\mathcal{O}))}^{p/2}
\\ &
\eqsim_{p,q}& \E_r \Biggl\|\sum_{n=1}^N r_n
G_n \Biggr\|_{L^p(\R_+\times\O;L^q(\mathcal{O};H)))}^p,
\end{eqnarray*}
where the last step follows by reversing the computation above.
\end{pf*}

\begin{remark}\label{rem:insuff}
The above proof uses the right-hand side inequality in Proposition \ref
{prop:NVW} in an
essential way; it seems that the simpler inequality~\eqref{eq:Mtype2} is insufficient for this purpose.
\end{remark}

%s4 ###
\section{\texorpdfstring{Proof of Theorem \protect\ref{thm:main}}{Proof of Theorem 1.1}}
\label{sec:maxreg1}

We start with a Poisson representation formula.

\begin{lemma}\label{lem:poisson}
Let $\alpha\in(0,\pi/2)$ and $\alpha'\in(\alpha, \pi]$ be given,
let $E$ be a Banach
space and let
$f\dvtx \Sigma_{\alpha'}\to E$ be a bounded analytic function. Then for
all $s>0$ we
have
\[
f(s) = \sum_{j\in\{-1,1\}} \frac{j}{2\alpha} \int_0^\infty
k_{\alpha}(u,s)
f(u e^{ij\alpha}) \, du,
\]
where $k_{\alpha}\dvtx  \R_+\times\R_+\to\R$ is given by
%
%e4.1 ###
\begin{equation}\label{eq:kalpha}
k_{\alpha}(u,t) = \frac{(t/u)^{\fracc{\pi}{2\alpha}}}{(t/u)^{\fraca
{\pi}{\alpha}}
+ 1} \frac{1}{u}.
\end{equation}
\end{lemma}

\begin{pf}
If $g\dvtx \Sigma_{\fraca12\pi+\varepsilon}\to E$ is analytic and bounded
for some $\varepsilon>0$,
then
\[
g(t) = \frac{1}{\pi} \int_{-\infty}^\infty\frac{t}{t^2 + v^2}
g(iv) \, dv
\]
by the Poisson formula on the half-space (see \cite{Ho}, Chapter 8).
For small $\varepsilon>0$ let $\phi\dvtx \Sigma_{\fracd12\pi+\varepsilon
}\to
\Sigma_{\alpha'}$ be defined by $\phi(z) :=
z^{\fraca{2\alpha}{\pi}}$. Then $\phi$ is analytic, and taking $g =
f\circ\phi$ gives
\[
f(t^{\fraca{2\alpha}{\pi}}) = \frac1\pi\int_{-\infty}^{\infty}
\frac{t}{t^2+v^2} f\bigl(|v|^{\fraca{2\alpha}{\pi}}e^{i \operatorname{sign}(v)
\alpha}\bigr) \, dv.
\]
The required result is obtained by taking $s= t^{\fraca{2\alpha}{\pi
}}$ and $u =
|v|^{\fraca{2\alpha}{\pi}}$.
\end{pf}

The next lemma isolates an elementary property of the functions $k_{\a}$.

\begin{lemma}\label{lem:kalpha2}
For $\alpha\in(0,\pi)$ and $\theta\in[0,1]$ put
\[
k_{\alpha,\theta}(u,t) := \sqrt{u}   (u/t)^{\theta} k_{\alpha}(u,t),
\]
where $k_{\alpha}$ is given by \eqref{eq:kalpha}. Then
\[
\sup_{u>0} \int_0^\infty\sqrt{t}  \biggl|\frac{\partial k_{\alpha
,\theta}}{\partial t}
(u,t) \biggr|\,dt <\infty.
\]
\end{lemma}

\begin{pf}
This is an easy consequence of the identity
\[
\frac{\partial k_{\alpha,\theta}}{\partial t} (u,t)
=u^{-3/2}  h'(t/u),
\]
where $h(x) = {x^{\fracc{\pi}{2\alpha}-\theta}}/(x^{\fraca{\pi
}{\alpha}}+1)$.
\end{pf}

Note that by Lemma \ref{lem:kalpha2}, small enough multiples of $
k_{\alpha,\theta}(u,\cdot)$ belong to the set $\mathcal{K}$ defined
in Section
\ref{sec:Lq}.

In the proof of Theorem \ref{thm:maxregintro} we shall need a
small generalization of Theorem~\ref{thm:main}, stated next as
Theorem \ref{thm:maxregLp}. Theorem \ref{thm:main}
corresponds to the special case $\theta=0$.

It will be useful to introduce the notation
\[
F\diamond G(t) := \int_0^t F(t-s)G(s)\,dW_H(s), \qquad t\ge0,
\]
whenever $F\dvtx \R_+\to\calL(L^q(\mathcal{O}))$ is a function
for which these stochastic integrals are well-defined in $L^q(\mathcal{O})$.
In order to see that the integrand
is well defined as an adapted $L^q(\mathcal{O};H)$-valued process we
note that
every bounded operator $T$ on $L^q(\mathcal{O})$ extends to a bounded operator
on $L^q(\mathcal{O};H)$ of the same norm (see \cite{Stein93}, Section~I.8.24);
on the dense subspace $L^q(\mathcal{O})\otimes H$
this extension is given by $T(f\otimes h) = Tf\otimes h$.

\begin{theorem}\label{thm:maxregLp} Let $q\in[2,\infty)$ and $\sigma
\in(0,\pi/2)$,
and suppose that $A$ has a bounded $H^\infty(\Sigma_\sigma)$-calculus
on $L^q(\mathcal{O})$. Let $S$ denote the bounded analytic semigroup
generated by
${-}A$. Set
\[
S_\theta(t) := \frac{t^{-\theta}}{\Gamma(1-\theta)} S(t).
\]
For all $p\in(2,\infty)$ and $\theta\in[0,\frac12)$ there exists a
constant $C\ge0$ such that
for all $G\in L_\F^p(\R_+\times\O;L^q(\mathcal{O};H))$ we have
$S_\theta
\diamond G(t)\in D(A^{\fraca12-\theta})$ almost surely for almost all
$t\ge0$ and
\[
\|A^{\fraca12-\theta} S_\theta\diamond
G\|_{L^p(\R_+\times\O;L^q(\mathcal{O}))}\leq C
\|G\|_{L^p(\R_+\times\O;L^q(\mathcal{O};H))}.
\]
This estimate also holds when $p=q=2$.
\end{theorem}

\begin{pf}
By a density argument it suffices to consider $\F$-adapted finite step
processes
$G\dvtx \R_+\times\O\to D(A_H)$, where $A_H = A\otimes I_H$ is the
generator of the bounded analytic semigroup
$S(t)$ viewed as acting on $L^q(\mathcal{O};H)$. For such~$G$,
the process $A_H^{\fraca12}G$ takes values in $L^q(\mathcal{O};H)$.
By H\"{o}lder's inequality and~\eqref{eq:Mtype2}, $S_\theta\diamond
G(t)$ and $A^{\fraca12-\theta}S_\theta\diamond G(t)$
are well defined in $L^q(\mathcal{O})$ for each $t\in\R_+$, and both
processes are jointly measurable on $\R_+\times\O$ (see \cite{NVW3}, Proposition~A.1).
\vspace*{1pt}

The idea of the
proof is to reduce the estimation of $A^{\fraca12-\theta} S_\theta
\diamond G$
to an estimation of $k_{\alpha,\theta}(u, \cdot)\diamond(V(u)G)$, where
$k_{\alpha,\theta}$ are the scalar kernels introduced in Section \ref
{sec:Lq} and
$V(u)$ is a suitable
operator depending on $u$ and $A$. The latter is then estimated using the
$H^\infty$-calculus of $A$.\vspace*{1pt}

Fix $\theta\in[0,\frac12)$. We proceed in two steps.

\textit{Step 1}.
First we shall rewrite $A^{\fraca12-\theta}S_\theta\diamond G$ using
Lemma \ref{lem:poisson}.
Fix\vspace*{1pt} $0<\a<\a'<\frac12\pi-\sigma$. Since $z\mapsto S(z)$ is
analytic and bounded on $\Sigma_{\a'}$, it follows from
Lemma~\ref{lem:poisson} that for all $x\in D(A^{\fraca12})$,
%
%e4.2 ###
\begin{equation}\label{eq:Spoisson}
 \quad \frac{1}{\Gamma(1-\theta)}(t-s)^{-\theta}A^{\fraca12-\theta}
S(t-s)x =
\int_0^\infty k_{\alpha,\theta}(u,t-s) V(u)x \, \frac{du}{u},
\end{equation}
where
\[
V(u) := \frac{1}{\Gamma(1-\theta)}\sum_{j\in\{-1,1\}} \frac
{j}{2\alpha}
(\varphi_j(u A))^2
\]
and $\varphi_{j}\in H^\infty_0(\Sigma_{\fracd12\pi-\alpha'})$ is
given by $\varphi_j(u) =
u^{\fraca14-\fracd12\theta} \exp(- \frac12 u e^{ij\alpha})$.
We write $I_{\alpha,\theta} = I(k_{\alpha,\theta})$ for the
operator as defined
by \eqref{eq:Ikoperator} with $k = k_{\alpha,\theta}$ as in Lemma~\ref{lem:kalpha2}.
By \eqref{eq:Spoisson} and the stochastic Fubini theorem we obtain,
for all $t\ge0$,
\begin{eqnarray*}
A^{\fraca12-\theta}S_\theta\diamond G(t)
& =&\int_0^t \int_0^\infty V(u)
k_{\alpha,\theta}(u,t-s) G(s) \, \frac{du}{u}\, d W_H(s)
\\ & =& \int_{0}^\infty V(u)   I_{\alpha,\theta}(u) G(t)\,
\frac{du}{u}.
\end{eqnarray*}

\textit{Step 2}.
Next we prove the estimate.
Let $E_1 = L^p(\R_+\times\O;L^q(\mathcal{O};H))$ and $E_2 =
L^p(\R_+\times\O;L^q(\mathcal{O}))$.
The space $L^q(\mathcal{O})$ is reflexive and therefore $E_2^* =
L^{p'}(\R_+\times\O;L^{q'}(\mathcal{O}))$ isometrically.
For all $\zeta^*\in
L^{p'}(\R_+\times\O;L^{q'}(\mathcal{O}))$ with $\frac1p+\frac
{1}{p'} = \frac1q+\frac{1}{q'} = 1$,
\begin{eqnarray*}
 &&  \langle A^{\fraca12-\theta}S_\theta\diamond G, \zeta^*\rangle
_{E_2} \\
&&  \qquad  =
\frac{1}{\Gamma(1-\theta)}\\
&&  \qquad \quad {}\times \sum_{j\in\{-1,1\}}
\frac{j}{2\alpha} \E\int_{0}^\infty    \biggl\langle\int
_{0}^\infty
(\varphi_j(u A))^2
I_{\alpha,\theta}(u)G(t) \, \frac{du}{u}, \zeta^*(t)
 \biggr\rangle_{L^q(\mathcal{O})} \, dt
\\ &&  \qquad =\frac{1}{\Gamma(1-\theta)}
\sum_{j\in\{-1,1\}} \frac{j}{2\alpha} \int_0^\infty
\langle
\varphi_j(u A)
I_{\alpha,\theta}(u) G, \varphi_j(u A^*)\zeta^*  \rangle_{E_2}
\frac{du}{u}
\\ &&  \qquad \stackrel{(*)}{=} \frac{1}{\Gamma(1-\theta)}
\sum_{j\in\{-1,1\}} \frac{j}{2\alpha}  \langle\varphi_j(u A)
I_{\alpha,\theta}(u) G, \varphi_j(u A^*)\zeta^*  \rangle
_{E_2(L^2(\R_+,\fraca{du}{u}))},
\end{eqnarray*}
where $\langle\cdot,\cdot\rangle_F$ denotes the duality pairing
between a
Banach space
$F$ and its dual $F^*$; the identity $(*)$ follows by writing out the
duality between the Banach function\vadjust{\goodbreak} spaces $E_2$ and $E_2^*$ as an
integral over $\R_+\times\O\times\mathcal{O}$ and then using the
Fubini theorem.
It follows that
\begin{eqnarray*}
|\langle A^{\fraca12-\theta}S_\theta\diamond G, \zeta^*\rangle|
& \leq&\frac{1}{\Gamma(1-\theta)}\sum_{j\in\{-1,1\}} \frac
{1}{2\alpha}
\|\varphi_j(u A) I_{\alpha,\theta}(u)G\|_{E_2(L^2(\R_+,\fraca{du}{u}))}
\\ &&\hphantom{\frac{1}{\Gamma(1-\theta)}\sum_{j\in\{-1,1\}}}
{}   \times\|\varphi_j(u A^*)\zeta^*\|_{E_2^*(L^2(\R
_+,\fraca{du}{u}))}.
\end{eqnarray*}
By Proposition \ref{prop:Hinftysquare} (applied ``pointwise'' in
$\R_+\times\O$),
\[
\|\varphi_j(u A^*)\zeta^*\|_{E_2^*(L^2(\R_+,\fraca{du}{u}))}
\leq C_1 \|\zeta^*\|_{E_2^*}.
\]
Since $\varphi_j(u A) I_{\alpha,\theta}(u)G =
I_{\alpha,\theta}(u)\varphi_j(u A_H) G$, from Proposition
\ref{prop:KW}
and another pointwise application of Proposition \ref{prop:Hinftysquare}
(this time for $A_H = A\otimes I_H$, noting that $A_H$
satisfies the assumptions of the proposition
if $A$ does) we obtain
\begin{eqnarray*}
 && \|\varphi_j(u A) I_{\alpha,\theta}(u)G\|_{E_2(L^2(\R_+,\fraca{du}{u}))}
\\ && \qquad  \leq R\bigl(I_{\alpha,\theta}(u)\dvtx  u\in\R\bigr) \|\varphi_j(u A_H)
G\|_{E_1(L^2(\R_+,
\fraca{du}{u}) )}
\\ && \qquad   \leq R\bigl(I_{\alpha,\theta}(u)\dvtx  u\in\R\bigr) C_2 \|G\|_{E_1}.
\end{eqnarray*}
By Lemma \ref{lem:kalpha2} the $R$-bound can be estimated by
\[
R\bigl(I_{\alpha,\theta}(u)\dvtx  u\in\R\bigr)\leq C_3 R(\mathcal{I}),
\]
and the latter is finite by Theorem \ref{thm:J-Lq}. We conclude that
\[
|\langle A^{\fraca12-\theta}S_\theta\diamond G, \zeta^*\rangle|
\leq\frac
{1}{\a\Gamma(1-\theta)} C_1 C_2 C_3
R(\mathcal{I}) \|G\|_{E_1}\|\zeta^*\|_{E_2^*}.
\]
Taking the supremum over all $\|\zeta^*\|_{E_2^*}\leq1$ it
follows that
\[
\|A^{\fraca12-\theta}S_\theta\diamond G\|_{E_2}\leq\frac{1}{\a
\Gamma(1-\theta)}C_1 C_2 C_3
R(\mathcal{I}) \|G\|_{E_1}.
\]
\upqed
\end{pf}

\begin{remark}\label{rem:fails}
As in \cite{Kry94}, Remark 2.1, one shows that the inequality in
Theorem \ref{thm:maxregLp}
fails for $p=q\in
[1,2)$.
In Section \ref{sec:counter} we prove that Theorem
\ref{thm:maxregLp} also fails for $p=2$ and $q\in(2,\infty)$.
\end{remark}

\begin{remark}\label{rem:MR-finiteT}
Stochastic maximal
$L^p$-regularity on bounded time intervals may be deduced from Theorem
\ref{thm:maxregLp} by considering processes $G$ with support in
$(0,T)\times\O$.
In this situation it suffices to know that $w+A$ has a~%
bounded $H^\infty$-calculus of angle less than $\pi/2$ for some $w\in
\R$ large enough,
and we obtain the inequality
\[
\|S\diamond G\|_{L^p((0,T)\times\O;D((w+A)^{\fraca12}))} \le
C e^{w T} \|G\|_{L^p((0,T)\times\O;L^q(\mathcal{O};H))}
\]
with the constant $C$ independent of $G$ and $T$. In particular,
injectivity of $A$ is not needed for this estimate. We leave the easy
details to the reader.\vspace*{-1pt}
\end{remark}

\begin{remark}\label{rem:Hilbert}
In the proof of Theorem \ref{thm:maxregLp} we used both inequalities of
Proposition \ref{prop:Hinftysquare}.
It is an open problem whether only the first one suffices.
For $p=q=2$ this is indeed the case.\vadjust{\goodbreak}
To see this, take $\varphi(z) = z^{1/2} \exp(-z)$ and assume that
\[
\|\varphi(tA) f\|_{L^2(\mathcal{O}; L^2(\R_+,\fraca{dt}{t}))} \leq
C\|f\|, \qquad f\in L^2(\mathcal{O}).
\]
Let $(h_n)_{n\geq1}$ be an orthonormal
basis for $H$. Then by the It\^{o} isometry, Fubini's theorem and the
Plancherel formula,
\begin{eqnarray*}
\|A^{\fraca12} S\diamond G\|_{L^{2}(\R_+\times\O;L^2(\mathcal{O}))}^2
& =&
\int_0^\infty\E\int_0^t \sum_{n\geq1} \|A^{\fraca12} S(t-s) G(s)
h_n\|^2_{L^2(\mathcal{O})} \, ds \, dt
\\ & =& \int_0^\infty\E\sum_{n\geq1} \int_0^\infty\|A^{\fraca
12} S(t) G(s)
h_n\|^2_{L^2(\mathcal{O})} \, dt \, ds
\\ & \leq& C^2 \int_0^\infty\E\sum_{n\geq1} \|G(s) h_n\|
^2_{L^2(\mathcal{O})}\,ds
\\ & =& C^2 \|G\|_{L^2(\R_+\times\O;L^2(\mathcal{O};H))}^2.
\end{eqnarray*}
\end{remark}

%s5 ###
\section{\texorpdfstring{Proof of Theorem \protect\ref{thm:maxregintro}}{Proof of Theorem 1.2}}
\label{sec:maxreg2}

To prepare for the proof of Theorem \ref{thm:maxregintro} we start by
collecting
some results on sums of closed operators on a UMD Banach space $E$;
below we shall only need the case $E= L^q(\mathcal{O})$ with $q\in
(1,\infty)$.

Let~$A$ have a bounded $H^\infty$-calculus on $E$ of angle less than
$\pi/2$.
Let~$\mathscr{A}$ be the closed and densely defined
operator on $L^p(\R_+;E)$ with domain $D(\mathscr{A}) :=
L^p(\R_+;D(A))$ defined by
\[
(\mathscr{A} f)(t) := A f(t).
\]
Let $\mathscr{B}$ be the closed and densely defined operator on
$L^p(\R_+;E)$
with domain $D(\mathscr{B}) := H^{1,p}_0(\R_+;E)$ given by
\[
\mathscr{B} f := f'.
\]
Here $H^{\theta,p}_0(\R_+;E) = \{f\in H^{\theta,p}(\R_+;E)\dvtx
f(0)=0\}$,
where $H^{\theta,p}(\R_+;E)=\break [L^p(\R_+;E), H^{1,p}(\R_+;E)]_{\theta}$
is the Bessel potential space defined by complex interpolation.

The operators $\mathscr{A}$ and $\mathscr{B}$ have bounded imaginary powers
(see, e.g., \cite{Am}, Lemma~III.4.10.5), and by \cite{PrSo}, Theorems 4 and 5),
the operator
\[
\mathscr{C}:=\mathscr{A}+ \mathscr{B}, \qquad D(\mathscr{C}) :=
D(\mathscr{A})\cap D(\mathscr{B}),
\]
is closed and has bounded imaginary powers as well.
Furthermore, $\mathscr{C}$ is injective and has dense range, and
for all $\theta\in(0,1)$ one has (see, e.g., \cite{Brz2}, Proposition 3.1)
%
%e5.1 ###
\begin{equation}\label{eq:Lambdainverse}
(\mathscr{C}^{-\theta} f)(t)= \frac{1}{\Gamma(\theta)}\int_0^t
(t-s)^{\theta-1} S(t-s) f(s) \, ds.
\end{equation}
Moreover, for all $\theta\in(0,1]$ one has
(see \eqref{eq:complexfract} and
\cite{Prussalshierboven}, Corollary~1)
%
%e5.2 ###
\begin{equation}\label{eq:idLambdalambda}
D(\mathscr{C}^{\theta})
= L^p(\R_+;D(A^{\theta})) \cap H^{\theta,p}_0(\R_+;E).\vadjust{\goodbreak}
\end{equation}

\begin{pf*}{Proof of Theorem \ref{thm:maxregintro}}
By a density argument, it suffices to consider an arbitrary
$\F$-adapted finite rank step process $G\dvtx \R_+\times\O\to D(A_H)$, where
$A_H = A \otimes I_H$.

(1)
By the Da Prato--Kwapie\'{n}--Zabczyk factorization argument (see \cite
{Brz2} and \cite{DPZ}, Section 5.3, and references therein), using
\eqref{eq:Lambdainverse}, the stochastic Fubini theorem and the
equality
\[
\frac1{\Gamma(\theta)
\Gamma(1-\theta)}\int_r^t (t-s)^{\theta-1} (s-r)^{-\theta} \, ds = 1
\]
one obtains, for all $t\in\R_+$,
\[
\mathscr{C}^{-\theta} (A^{\fraca12-\theta} S_\theta\diamond G)(t)
= A^{\fraca12-\theta}S\diamond G(t)    \qquad \mbox{almost surely},
\]
and hence, by \eqref{eq:idLambdalambda} and Theorem \ref{thm:maxregLp},
\begin{eqnarray*}
\| A^{\fraca12-\theta} S\diamond G\|_{L^p(\O;H^{\theta,p}(\R
_+;L^q(\mathcal{O})))}
& \le&
\|\mathscr{C}^{\theta} (A^{\fraca12-\theta}S\diamond G)
(t)\|_{L^p(\R_+\times\O;L^q(\mathcal{O}))}
\\ & =&
\|(A^{\fraca12-\theta}S_\theta\diamond G)
(t)\|_{L^p(\R_+\times\O;L^q(\mathcal{O}))} \\
&
\leq& C\|G\|_{L^p(\R_+\times\O;L^q(\mathcal{O}
;H))}.
\end{eqnarray*}

(2) Let $\theta\in(\frac1p,\frac12)$. By \cite{Zacher05}, Theorem 3.6
(see also \cite{Am}, Theorem III.4.10.2, and
\cite{Prussalshierboven}, Proposition 3) there is a continuous embedding
\[
H^{\theta,p}(\R_+;L^q(\mathcal{O})) \cap L^p(\R_+;D(A^{\theta}))
\hookrightarrow \mathit{BUC}\biggl(\R_+;D_A\biggl(\theta-\frac1p,p\biggr)\biggr)
\]
of norm $K$. Here $\mathit{BUC}(\R_+;L^q(\mathcal{O}))$ denotes the Banach
space of all bounded uniformly continuous functions
from $\R_+$ to $L^q(\mathcal{O})$.
Combining this with the result of part (1), noting that
$\| S\diamond G\|_{L^p(\R_+\times\O;D(A^{\fraca12}))}  \leq  C
\|G\|_{L^p(\R_+\times\O;L^q(\mathcal{O};H))}$
by Theorem \ref{thm:main} and the fact that $0\in\varrho(A)$,
\begin{eqnarray*}
 && \|A^{\fraca12-\theta}S\diamond G\|_{L^p(\O;\mathit{BUC}(\R_+;D_A(\theta
-\fraca1p,p)))}
\\ && \qquad  \leq K \max
 \bigl\{
\|A^{\fraca12-\theta}S\diamond G\|_{L^p(\O;H^{\theta,p}(\R
_+;L^p(\mathcal{O})))},\\
&& \hphantom{\qquad\leq K \max
 \bigl\{}\hspace*{9.5pt}
\|A^{\fraca12-\theta}S\diamond G\|_{L^p(\O;L^p(\R_+;D(A^{\theta
})))} \bigr\}
\\ && \qquad  \leq C K\|G\|_{L^p(\R_+\times\O;L^q(\mathcal{O};H))}.
\end{eqnarray*}
Hence by \cite{Tr1}, Theorem 1.15.2(e),
\begin{eqnarray*}
&&\|S\diamond G\|_{L^p(\O;\mathit{BUC}(\R_+;D_A(\fraca12-\fraca1p,p)))} \\
&& \qquad \eqsim
_{A,\theta,p}
\|A^{\fraca12-\theta}S\diamond G\|_{L^p(\O;\mathit{BUC}(\R_+;D_A(\theta
-\fraca1p,p)))} \\ &&  \qquad  \leq CK
\|G\|_{L^p(\R_+\times\O;L^q(\mathcal{O};H))}.
\end{eqnarray*}
\upqed
\end{pf*}

\begin{remark}
A standard stopping time argument (see, e.g., \cite{DPZ}, Proposition~4.16)
shows that Theorem
\ref{thm:maxregintro} can be localized. For instance,
from Theorem~\ref{thm:maxregintro}(2) one can\vadjust{\goodbreak} infer that for all
$G\in L_{\F}^0(\O;L^p(\R_+;L^q(\mathcal{O};H)))$ one has
\[
S\diamond G\in L^0\biggl(\O;\mathit{BUC}\biggl(\R_+;D_A\biggl(\frac12-\frac1p, p\biggr)\biggr)\biggr).
\]
Here $L^0(\O;E)$ denotes the space of strongly measurable functions on
$\O$ with values in a Banach space $E$.
\end{remark}

\begin{remark} Arguing as in Remark \ref{rem:MR-finiteT},
also Theorem \ref{thm:maxregintro} admits a~version for bounded time intervals.
\end{remark}

\begin{remark} As has been pointed out in Remark \ref{rem:UMD},
the role of $L^q(\mathcal{O})$ in Proposition \ref{prop:NVW}
can be taken over by an arbitrary UMD Banach space~$E$.
We do not know, however, whether Theorems \ref{thm:J-Lq} and
Proposition \ref{prop:FS} can be extended
to UMD Banach spaces $E$, say with (martingale) type $2$ in order to
rule out the spaces
$L^q(\mathcal{O})$ with $q\in(1,2)$ for which Theorem \ref{thm:main}
is known to be false (see Remark~\ref{rem:fails}).
\end{remark}

%s6 ###
\section{A counterexample to stochastic maximal $L^2$-regularity}\label{sec:counter}

We show next that Theorem \ref{thm:main} is not valid with
$p=2$ and $q\in(2,\infty)$, even when $H=\R$ and $G$ is deterministic.
Stated differently, analytic generators on $L^q(\mathcal{O})$ do
not always enjoy stochastic
maximal $L^2$-regularity for $q\in(2,\infty)$. This
is rather surprising, since stochastic maximal $L^2$-regularity for
Hilbert spaces [in particular, for $L^2(\mathcal{O})$] is easy
to prove (see Remark \ref{rem:Hilbert}).

In the next theorem, $W$ denotes a real-valued Brownian motion.

\begin{theorem}
Let $q\in(2,\infty)$ and
fix an increasing sequence $0<\l_1< \l_2<\cdots $ diverging to $\infty$.
Let $A$ be the diagonal operator on $\ell^q$ defined by $A e_k :=
\lambda_k e_k$ with its maximal domain.
Then $A$ has a bounded $H^\infty$-calculus of zero angle, but there
does not exist a constant $C$ such that for
all $G\in L^2(\R_+;\ell^q)$,
%
%e6.1 ###
\begin{equation}\label{eq:assumpstoch}
\int_0^\infty\E \biggl\|\int_0^t A^{\fraca12} S(t-s) G(s) \, d
W(s) \biggr\|^2_{\ell^q} \, dt \leq C^2 \int_0^\infty\|G(t)\|^2_{\ell
^q} \, dt.
\end{equation}
\end{theorem}

\begin{pf}
The verification that $A$ has a bounded $H^\infty$-calculus of zero
angle is
routine.

By Proposition \ref{prop:NVW}, the estimate \eqref{eq:assumpstoch} is
equivalent to
%
%e6.2 ###
\begin{eqnarray}\label{eq:stochest}
&&\E\int_0^\infty \biggl(\sum_{k\geq1}
\biggl (\int_0^t \lambda_k e^{-2\lambda_k(t-s)}   |g_k(s)|^2 \,
ds \biggr)^{q/2} \biggr)^{2/q} \, dt\nonumber
\\[-8pt]
\\[-8pt]
&& \qquad  \leq C^2_1 \int_0^\infty
\biggl(\sum_{k\geq
1}|g_k(t)|^q \biggr)^{2/q} \, dt,
\nonumber
\end{eqnarray}
where $C_1$ is a different constant independent of $G = (g_k)_{k\geq1}$.
We claim that this inequality implies deterministic maximal
$L^{1}$-regularity for
the operator $B = 2A$ on the space $\ell^{q/2}$, by which we mean that
there is a constant~$C_2$ such that
\[
 \biggl\|\int_0^t B e^{-(t-s)B} f(s) \, ds \biggr\|_{L^{1}(\R_+;\ell
^{q/2})}\leq
C_2 \|f\|_{L^{1}(\R_+;\ell^{q/2})}
\]
for all $f = (f_k)_{k\ge1}$ in $L^{1}(\R_+;\ell^{q/2})$. This
inequality is equivalent to
%
%e6.3 ###
\begin{eqnarray}\label{eq:detest}
&& \int_0^\infty \biggl(\sum_{k\geq
1}  \biggl|\int_0^t 2\lambda_k   e^{-2\lambda_k(t-s)} f_k(s) \,
ds \biggr|^{q/2} \biggr)^{2/q} \, dt\nonumber
\\[-8pt]
\\[-8pt] && \qquad  \leq C_2 \int_0^\infty
 \biggl(\sum_{k\geq1}|f_k(s)|^{q/2}  \biggr)^{2/q} \, ds.
\nonumber
\end{eqnarray}
To see that \eqref{eq:detest} follows from \eqref{eq:stochest}, we
may reduce to
nonnegative $f$ by considering positive and negative parts of each
$f_k$ separately.
Then \eqref{eq:detest} follows by taking $g_k = \sqrt{f_k^{\pm}}$ in
\eqref{eq:stochest}.

Now the theorem follows from
\cite{Guerre}, where it is shown that $B$
fails maximal $L^1$-regularity on $\ell^{q/2}$ with $q\in(2,\infty)$.
\end{pf}

Of course, by Theorem \ref{thm:main} the operator $A$ of this example
has stochastic maximal
$L^p$-regularity for $p\in(2,\infty)$.

%s7 ###
\section{Discussion}\label{sec:discussion}
We have already compared Theorem \ref{thm:main} with Krylov's inequalities
\eqref{eq:Kry} and \eqref{eq:Kry-pq} in the \hyperref[sec1]{Introduction}.
Theorem \ref{thm:main} also extends various other regularity results
in the
literature.

%s7.1 ###
\subsection{Hilbert spaces}
For generators ${-}A$ of bounded strongly continuous analytic semigroups
on Hilbert
spaces, stochastic maximal $L^2$-regularity was proved by Da Prato
(see \cite{DPZ}, Section 6.3, and references therein)
under the assumption
$D(A^{\theta}) = D_A(\theta, 2)$ for all $\theta\in(0,1).$
This condition is equivalent to the existence of a bounded $H^\infty$-calculus
of angle less than $\frac12\pi$ for $A$ (see~\cite{Haase2}, Remark~6.6.10,
and \cite{KuWe}, Theorem 11.9). Thus, the case $p=q=2$ of
Theorem \ref{thm:main} contains Da Prato's result (see also Remark~\ref{rem:Hilbert}). For $p\in(2,\infty)$, Theorem
\ref{thm:main} seems to be new even in the Hilbert space case, that
is, $q=2$.
Similarly, only the case $p=2$ of Theorem \ref{thm:maxregintro}(2)
is known in the Hilbert space case (see \cite{DPZ}, Section~6.2, and note that $D_A(\frac12,2)=D(A^{1/2})$ when $E$ is a
Hilbert space;
here we should mention the fact that analytic contraction semigroups
on Hilbert spaces have a bounded $H^\infty$-calculus of angle less
than $\pi/2$
\cite{KuWe}, Theorem 11.13).
As observed in \cite{Fla92}, the above mentioned assumption in Da
Prato's result can be\vspace*{1pt} weakened to $D(A^{\fraca12}) \supseteq D_A(\frac
12, 2)$.

%s7.2 ###
\subsection{Martingale type $2$ spaces}\label{subsec:discussion}
For $p=2$, the related estimate
%
%e7.1 ###
\begin{equation}\label{eq:maxregineqintro}
\E\|U\|_{L^p(0,T;D(A))}^2\le C^2
\E\|G\|_{L^p(0,T;D_{A_H}(\fraca12,2))}^2
\end{equation}
was obtained by Brze\'{z}niak \cite{Brz1} for so-called $M$-type $2$
Banach spaces
$E$ [this class includes the $L^q$-spaces for $q\in[2,\infty)$].
If $A$ is a second order elliptic operator on a space $E = L^q(\R^d)$,
one typically
has
\[
D_{A_H}\bigl(\tfrac12,2\bigr) =
B_{q,2}^1(\R^d;H) \subseteq
H^{1,q}(\R^d;H)
= D(A_H^\fraca12),
\]
where the middle inclusion, being a consequence of
\cite{Tr1}, Remark~2.3.3/4 and Theorem~4.6.1, is strict for $q\in(2,\infty)$.
Similar reasoning applies in the case of bounded regular domains in $\R^d$.

As a consequence, the inequality \eqref{eq:maxregineqintro} is weaker than
the one which follows from \eqref{eq:delta} (with $\delta = \frac12$),
\begin{equation}\label{eq:ours} \E \|U\|_{L^p(0,T;D(A))}^2\le
C^2\E\|G\|_{L^p(0,T;D(A_H^\fraca12))}^2.
\end{equation}
More importantly, the fact that the real interpolation spaces
$D_{A_H}(\tfrac12,2)$ are Besov spaces causes difficulties in the treatment
of nonlinear problems (as was noted in \cite{Brz1,Kry}). Such problems can
be avoided if one uses the inequality~\eqref{eq:ours} instead.

%s7.3 ###
\subsection{Real interpolation spaces}
For analytic generators ${-}A$ on $M$-type $2$ spaces $E$, stochastic maximal
$L^p$-regularity for $p\in[2,\infty)$ in the
real interpolation spaces $D_A(\theta, p)$ for $\theta\in[0,1)$ was
proved by
Da Prato and Lunardi~\cite{DPL} (see also~\cite{BrzHau});
the solution $U$ then belongs to
$L^p(0,T;{D}_A(\theta+\frac12,p))$.
With $\theta=\frac12$ this gives the estimate
%
%e7.3 ###
\begin{equation}\label{eq:DPL}
\E\|U\|_{L^p(0,T;D_A(1,p))}^2\le C^2\E\|G\|_{L^p(0,T;D_{A_H}(\fraca12,p))}^2.
\end{equation}
Comparing with \eqref{eq:ours},
this time the applicability is limited by
the observation that the solution space
$D_A(1,p)$ may be larger than $D(A)$ when $E=L^q(\mathcal{O})$ with
\mbox{$q\ge p$}. This happens, for instance,
in the special case where~$A$ is a second order elliptic operator on
$E=L^q(\R^d)$
with $q\in(2,\infty)$. Taking $p=q$ one has
\[
D_A(1,q) = B^2_{q,q}(\R^d)
\supsetneqq H^{2,q}(\R^d) = D(A).
\]
Again similar reasoning applies in the case of bounded regular domains in~$\R^d$.

When comparing the results of \cite{DPL} with ours, it should be noted
that if
${-}A$ is invertible and generates a bounded strongly continuous analytic
semigroup
on a Banach space $E$, then by a result of Dore (see, e.g.,
\cite{Haase2}, Corollary 6.5.8) $A$
admits a bounded $H^\infty$-functional calculus of angle less that
$\pi/2$
on the spaces $D_A(\theta,p)$
for all $\theta\in(0,1)$ and $p\in(1,\infty)$. Hence, at least for
$E=L^q(\mathcal{O})$, estimate \eqref{eq:DPL} is also contained in
Theorem \ref{thm:main}.

\section*{Acknowledgments}
We thank Wolfgang Desch, Stig--Olof Londen and the anonymous referee
for careful reading.

% imsref loaded by smiklovaite, 2011-01-31 12:43:39
% imsref loaded by smiklovaite, 2011-01-31 12:53:59
%

\printaddresses

\end{document}